\documentclass[11pt]{amsart}%
\usepackage{amsmath}
\usepackage{amsfonts}
\usepackage{geometry}
\usepackage{graphicx}
\usepackage{amssymb}
\usepackage{epstopdf}%
\providecommand{\U}[1]{\protect\rule{.1in}{.1in}}
\newtheorem{theorem}{Theorem}[section]

\newtheorem{lemma}[theorem]{Lemma}
\newtheorem{proposition}[theorem]{Proposition}
\newtheorem{definition}[theorem]{Definition}

\newtheorem{remark}[theorem]{Remark}
\numberwithin{equation}{section}
\begin{document}
\title{\textbf{On the eigenvalue process of a matrix} \textbf{fractional Brownian
motion} }
\author{David Nualart}
\address{Department of Mathematics\\
University of Kansas \\
405 Snow Hall, Lawrence\\
Kansas 66045-2142, USA\\
nualart@math.ku.edu}
\author{Victor P\'{e}rez-Abreu}
\address{Department of Probability and Statistics\\
Center for Research in Mathematics CIMAT\\
Apdo. Postal 402, Guanajuato, Gto. 36000, Mexico\\
pabreu@cimat.mx}
\date{}
\maketitle

\begin{abstract}
We investigate the process of eigenvalues of a symmetric matrix-valued process
which upper diagonal entries are independent one-dimensional H\"{o}lder
continuous Gaussian processes of order $\gamma\in(1/2,1)$. Using the
stochastic calculus with respect to the Young's integral we show that these
eigenvalues do not collide at any time with probability one. When the matrix
process has entries that are fractional Brownian motions with Hurst parameter
$H\in(1/2,1)$,\ we find a stochastic differential equation in a Malliavin
calculus sense for the eigenvalues of the corresponding matrix fractional
Brownian motion. A new generalized version of the It\^{o} formula for the
multidimensional fractional Brownian motion is first established.\medskip



\end{abstract}

\footnote{David Nualart was partially supported by the NSF grant DMS1208625.}

\section{Introduction}

In a pioneering work in 1962, the nuclear physicist Freeman Dyson \cite{Dy62}
studied the stochastic process of eigenvalues of an Hermitian matrix Brownian
motion. The case of the (real) symmetric matrix Brownian motion was first
considered by Mc Kean \cite{Mc69} in 1969. In both cases the corresponding
processes of eigenvalues are called Dyson Brownian motion and are governed by
a noncolliding system of It\^{o} Stochastic Differential Equations (SDEs) with
non-smooth diffusion coefficients; see the modern treatments, for example, in
the books by Anderson, Guionnet and Zeitouni \cite{AGZ09} and Tao \cite{Ta12}.

\medskip More specifically, for the symmetric case, let $\left\{
B(t)\right\}  _{t\geq0}=\{(b_{jk}(t))\}_{t\geq0}$ be a $d\times d$ symmetric
matrix Brownian motion. That is, $(b_{jj}(t))_{j=1}^{d},(b_{jk}(t))_{j<k},$ is
a set of $d(d+1)/2$ independent one-dimensional Brownian motions with
parameter $(1+\delta_{jk})t.$ For each $t>0$, $B(t)$ is a Gaussian Orthogonal
(GO) random matrix with parameter $t$ (\cite{AGZ09}, \cite{Me04}). Let
$\lambda_{1}(t)\geq\lambda_{2}(t)\geq\cdots\geq\lambda_{d}(t),$ $t\geq0,$ be
the $d$-dimensional stochastic process of eigenvalues of $B$. If the
eigenvalues start at different positions $\lambda_{{\small 1}}(0)>\lambda
_{{\small 2}}(0)>\cdots>\lambda_{{\small d}}(0)$, then they never meet at any
time ($\lambda_{1}(t)>\lambda_{2}(t)>\cdots>\lambda_{d}(t)$ almost surely
$\forall t>0$) and furthermore they form a diffusion process satisfying the
It\^{o} SDE
\begin{equation}
\lambda_{i}(t)=\sqrt{2}W_{t}^{i}+\sum_{j\neq i}\int_{0}^{t}\frac{\mathrm{d}%
s}{\lambda_{i}(s)-\lambda_{j}(s)},\quad1\leq i\leq d,\quad t\geq0, \label{21}%
\end{equation}
where $W_{t}^{1},\dots,W_{t}^{d}$ are independent one-dimensional standard
Brownian motions.

\medskip Thus, Dyson Brownian motion can be thought as a model for the
evolution of $d $ Brownian motions $\lambda_{1}(t),\dots,\lambda_{d}(t)$ that
are restricted to never intersect and having a repulsion force which is
inversely proportional to the distance between any two eigenvalues. For that
reason $\lambda(t)=(\lambda_{1}(t),\dots, \lambda_{d}(t))$ is also called a
Dyson non-colliding Brownian motion.

\medskip Different aspects of the Dyson Brownian motion have been considered
by several authors \cite{Be08}, \cite{CL97}, \cite{Ch92}, \cite{Ga99},
\cite{Gu08}, \cite{KT13}, \cite{Is01}, \cite{PaT07}, \cite{RS93}. The
corresponding eigenvalues processes of other matrix stochastic processes and
their associated It\^{o} SDEs have been studied in \cite{Br89}, \cite{Br91},
\cite{De07}, \cite{De07b}, \cite{KT04}, \cite{KO01}, \cite{PaT09}, among others.

\medskip In the literature, sometimes a Dyson process means a $d\times d$
matrix process in which the entries undergo diffusion, see \cite{TW04}. The
evolution of the singular values of a matrix process is often called Laguerre
process, being an important example the Wishart case (\cite{Br89},
\cite{Br91}, \cite{De07}, \cite{KO01}, \cite{PaT09}).

\medskip The purpose of this paper is to study the eingenvalue process of a
symmetric matrix Gaussian process $\left\{  G(t)\right\}  _{t\geq0}$ which
entries in the upper diagonal part are independent one-dimensional zero mean
Gaussian processes with H\"{o}lder continuous paths of order $\gamma
\in(1/2,1).$ In this case $G(t)$ is still a $GO$ random matrix of parameter
$t$, for each $t>0,$ and therefore the eigenvalues of $G(t)$ are distinct for
each $t>0$ a.s. However $\left\{  G(t)\right\}  _{t\geq0}$ is not a matrix
process with independent increments nor a matrix diffusion. We prove that the
corresponding eigenvalues do not collide at any time with probability one. The
proof of this fact is based on the stochastic calculus with respect to the
Young's integral. Our noncolliding method of proof is different from the case
of the Brownian matrix and other matrix diffusions, for which one first has to
find a diffusion process governing the eigenvalue process as done, for
example, in \cite{AGZ09} for the Hermitian Brownian motion. Our result does
not include the latter as a special case.

We study in detail the case when the matrix process has entries that are
fractional Brownian motions with Hurst parameter $H\in(1/2,1)$. For this
matrix process we find a stochastic differential equation for the eigenvalue
process, similar to equation (\ref{21}), where instead of $\sqrt{2}W_{t}^{i}$
we obtain processes $Y_{t}^{i}$ which are expressed as Skorohod indefinite
integrals. This equation is derived applying a generalized version of the
It\^{o} formula in the Skorohod sense for the multidimensional fractional
Brownian motion, which has its own interest. Each process $Y^{i}$ has the same
$H$-self-similarity and $1/H$-variation properties as a one-dimensional
fractional Brownian motion, although there is no reason for them to be
fractional Brownian motions. This phenomenon is related to the representation
of fractional Bessel processes, established by Guerra and Nualart in
\cite{GN05}. In the case of the fractional Bessel process, the fact that the
indefinite Skorohod integral appearing in the representation is not a
one-dimensional standard Brownian motion was proved by Hu and Nualart in
\cite{HN05}. In our case we conjecture that each $Y^{i}$ is not a
one-dimensional fractional Brownian motion, but at this moment we are not able
to give a proof of this fact.

The paper is organized as follows. Section 2 contains preliminaries on
Malliavin calculus for the fractional Brownian motion. Section 3 establishes a
generalized version of the It\^{o} formula for the multidimensional fractional
Brownian motion, in the case of functions that are smooth only on a dense
subset of the Euclidean space, as needed to describe the evolution of the
eigenvalues, a result does not covered in the literature. It also recalls a
property on the $1/H$-variation of a divergence integral and establishes a
result that allows to compute the $1/H$-variation of a Skorohod integral wrt a
multidimensional Brownian motion. In Section 4 we prove that if the matrix
process $\left\{  G(t)\right\}  _{t\geq0}$ has entries that are H\"{o}lder
continuous Gaussian processes of order $\gamma\in(1/2,1)$, then the
corresponding eigenvalues processes do not collide at any time with
probability one. Finally, Section 4 considers the SDE of the eigenvalues of a
matrix fractional Brownian notion similar to (\ref{21}) but in the Malliavin
calculus sense.

\section{Malliavin Calculus for the fBm}

In this section we present some basic facts on the Malliavin calculus, or
stochastic calculus of variations, with respect to the fractional Brownian
motion. We refer the reader to \cite{Nu06} for a detailed account of this topic.

Suppose that $B=\{B_{t}, t\ge0\}$ is a fractional Brownian motion with Hurst
parameter $H\in(1/2,1)$. That is, $B$ is a zero mean Gaussian process with
covariance
\[
R(t,s)= \mathbb{E}(B_{t}B_{s}) =\frac12 \left(  t^{2H} + s^{2H} - |t-s|^{2H}
\right)  .
\]
The process $B$ is $H$-self-similar, that is, for any $a>0$, $\{B_{at},
t\ge0\}$ and $\{a^{H} B_{t}, t\ge0\}$ have the same law. On the other hand, it
possesses a finite $1/H$-variation on any time interval of length $t$ (see
Definition \ref{def1}) equals to $t \mathbb{E} (|Z|^{1/H})$, where $Z$ is a
$N(0,1)$ random variable.

\medskip Fix a time interval $[0,T]$ and let $\mathcal{H}$ be the Hilbert
space defined as the closure of the set of step functions with respect to the
scalar product
\[
\langle\mathbf{1}_{[0,t]} ,\mathbf{1}_{[0,s]} \rangle_{\mathcal{H}} =R(s,t).
\]
Given two step functions $\varphi, \psi$ on $[0,T]$, its inner product can be
expressed as
\[
\langle\varphi, \psi\rangle_{\mathcal{H}}= \alpha_{H} \int_{0}^{T} \int%
_{0}^{T} \varphi_{t} \psi_{s} |t-s|^{2H-2} dtds,
\]
where $\alpha_{H} =H(2H-1)$. We denote by $|\mathcal{H}|$ the space of
measurable functions $\varphi$ on $[0,T]$ such that
\[
\| \varphi\|^{2}_{|\mathcal{H}|} := \alpha_{H}\int_{0}^{T} \int_{0}^{T}
|\varphi_{t}| |\varphi_{s} | |t-s|^{2H-2} dtds<\infty.
\]
This space is a Banach space, which is isometric to a subspace of
$\mathcal{H}$ and it will be identified with this subspace. Moreover, we have
the following continuous embeddings (see \cite{MeMiVa01}).
\begin{equation}
\label{tt2}L^{1/H}([0,T]) \subset|\mathcal{H}| \subset\mathcal{H}.
\end{equation}
The mapping $\mathbf{1}_{[0,t]} \mapsto B_{t}$ can be extended to a linear
isometry between $\mathcal{H}$ and the Gaussian space generated by $B$. We
denote this isometry by $h\mapsto B(h)$.

Let $\mathcal{S}$ be the set of smooth and cylindrical random variables of the
form
\[
F=f\left(  B(h _{1}),\ldots,B(h _{n})\right)  \text{,}%
\]
where $n\geq1$, $f\in C_{b}^{\infty}\left(  \mathbb{R}^{n}\right)  $ ($f$ and
all its partial derivatives are bounded) and $h _{i}\in\mathcal{H}$. The
derivative operator is defined in $\mathcal{S}$ as the $\mathcal{H}$-valued
random variable
\[
DF:=\sum_{i=1}^{n}\frac{\partial f}{\partial x_{i}}\left(  B(h _{1}%
),\ldots,B(h_{n})\right)  h _{i}.
\]
The derivative operator is a closable operator from $L^{p}\left(
\Omega\right)  $ into $L^{p}\left(  \Omega;\mathcal{H}\right)  $ for any
$p\geq1$. We denote by $D^{k}$ the iteration of $D$. For any $p\geq1$, we
define the Sobolev space $\mathbb{D}^{k,p}$ as the closure of $\mathcal{S}$
with respect to the norm
\[
\left\|  F\right\|  _{k,p}^{p}:=\mathbb{E}\left[  \left|  F\right|
^{p}\right]  +\mathbb{E}\left[  \sum_{j=1}^{k}\left\|  D^{j}F\right\|
_{\mathcal{H}^{\otimes j}}^{p}\right]  .
\]
In a similar way, given a Hilbert space $V$, we can define the Sobolev space
of $V$-valued random variables, denoted by $\mathbb{D}^{k,p}(V)$.

The divergence operator $\delta$ is the adjoint of the derivative operator,
defined by means of the duality relationship
\begin{equation}
\label{dua}\mathbb{E}\left[  F\delta(u)\right]  =\mathbb{E}\left[
\left\langle DF,u\right\rangle _{\mathcal{H}}\right]  ,
\end{equation}
where $u$ is a random variable in $L^{2}\left(  \Omega;\mathcal{H}\right)  $.
The domain $\mathrm{Dom}\left(  \delta\right)  $ is the set of random
variables $u\in L^{2}\left(  \Omega;\mathcal{H}\right)  $ such that
\[
\left|  \mathbb{E} \left(  \left\langle DF,u\right\rangle _{\mathcal{H}}
\right)  \right|  \leq c\left\|  F\right\|  _{2}%
\]
for all $F\in\mathbb{D}^{1,2}$. We have the inclusion $\mathbb{D}%
^{1,2}(\mathcal{H})\subset$\textrm{Dom}$\left(  \delta\right)  $, and for any
$u\in\mathbb{D}^{1,2}(\mathcal{H})$, the variance of the divergence of $u$ can
be computed as follows
\[
\mathbb{E}\left[  \delta(u)^{2}\right]  = \mathbb{E}\left[  \left\|
u\right\|  _{\mathcal{H}}^{2}\right]  +\mathbb{E}\left[  \left\langle
Du,\left(  Du\right)  ^{\ast}\right\rangle _{\mathcal{H\otimes H}}\right]
\leq\left\|  u\right\|  _{\mathbb{D}^{1,2}(\mathcal{H})}^{2},
\]
where $\left(  Du\right)  ^{\ast}$ is the adjoint of $Du$ in the Hilbert space
$\mathcal{H\otimes H}.$

By Meyer's inequalities, for all $p>1$, the divergence operator can be
extended to the space $\mathbb{D}^{1,p}(\mathcal{H})$ and the mapping
\[
\delta:\mathbb{D}^{1,p}(\mathcal{H})\rightarrow L^{p}\left(  \Omega\right)
\]
is a continuous operator, that is,
\begin{equation}
\label{tt3}\left\|  \delta(u)\right\|  _{p}\leq C_{p}\left\|  u\right\|
_{\mathbb{D}^{1,p}(\mathcal{H)}}.
\end{equation}

Denote by $\left|  \mathcal{H}\right|  \mathcal{\otimes}\left|  \mathcal{H}%
\right|  $ the space of measurable functions $\varphi$ defined on $\left[
0,T\right]  ^{2}$ such that
\[
\left\|  \varphi\right\|  _{\left|  \mathcal{H}\right|  \mathcal{\otimes
}\left|  \mathcal{H}\right|  }^{2}:=\alpha_{H}^{2}\int_{\left[  0,T\right]
^{4}}\left|  \varphi_{r,\theta}\right|  \left|  \varphi_{u,\eta}\right|
\left|  r-u\right|  ^{2H-2}\left|  \theta-\eta\right|  ^{2H-2}\mathrm{d}%
r\mathrm{d}u\mathrm{d}\theta\mathrm{d}\eta<\infty.
\]
As in the case of functions of one variable, this space is a Banach space,
which is isometric to a subspace of $\mathcal{H\otimes H}$ and it will be
identified with this subspace. For any $p\geq1$, we denote by $\mathbb{D}%
^{1,p}(\left|  \mathcal{H}\right|  ) $ the subspace of the Sobolev space
$\mathbb{D}^{1,p}(\mathcal{H})$ whose elements $u$ are such that $u\in\left|
\mathcal{H}\right|  $ a.s., $Du\in\left|  \mathcal{H}\right|  \mathcal{\otimes
}\left|  \mathcal{H}\right|  $ a.s. and
\[
\mathbb{E}\left[  \left\|  u\right\|  _{\left|  \mathcal{H}\right|  }%
^{p}\right]  +\mathbb{E}\left[  \left\|  Du\right\|  _{\left|  \mathcal{H}%
\right|  \mathcal{\otimes}\left|  \mathcal{H}\right|  }^{p}\right]
<\infty\text{.}%
\]

Then, using the embedding (\ref{tt2}) one can show that for any $p>1$, we have
the continuous embedding $\mathbb{L}_{H}^{1,p} \subset\mathbb{D}^{1,p}(\left|
\mathcal{H}\right|  ) $, where $\mathbb{L}_{H}^{1,p}$ is the set of processes
$u\in\mathbb{D}^{1,p}(\left|  \mathcal{H}\right|  \mathcal{)} $ such that
\begin{equation}
\left\|  u\right\|  _{\mathbb{L}_{H}^{1,p}}^{p}:=\mathbb{E}\left[  \left\|
u\right\|  _{L^{1/H}\left(  \left[  0,T\right]  \right)  }^{p}\right]
+\mathbb{E}\left[  \left\|  Du\right\|  _{L^{1/H}\left(  \left[  0,T\right]
^{2}\right)  }^{p}\right]  <\infty\text{.} \label{Lnorm}%
\end{equation}
As a consequence, Meyer inequalities (\ref{tt3}) imply
\begin{equation}
\left\|  \delta(u)\right\|  _{p}\leq C^{\prime}_{p}\left\|  u\right\|
_{\mathbb{L}_{H}^{1,p}} \label{estimatedelta}%
\end{equation}
if $u\in\mathbb{L}_{H}^{1,p}$ and $p>1$.

\medskip Let $N\in\mathbb{N}$, $d\geq2$, $H\in\left(  1/2,1\right)  \, $\ and
consider the $N$-dimensional fractional Brownian motion
\[
B=\{B_{t} , t\ge0\} =\left\{  B_{t}^{(1)},\ldots,B_{t}^{(N)}, t\in
[0,T]\right\}  ,
\]
with Hurst parameter $H,$ defined on the probability space $\left(
\Omega,\mathcal{F},P\right)  $, where $\mathcal{F}$ is generated by $B$. That
is, the components $B^{\left(  i\right)  },$ $i=1,\ldots d,$ are independent
fractional Brownian motions with Hurst parameter $H$. We can define the
derivative and divergence operators, $D^{\left(  i\right)  }$ and
$\delta^{\left(  i\right)  }$, with respect to each component $B^{\left(
i\right)  }$, as before. Denote by $\mathbb{D}_{i}^{1,p}\left(  \mathcal{H}%
\right)  $ the associated Sobolev spaces. We assume that these spaces include
functionals of all the components of $B$ and not only of component $i$.
Similarly, we introduce the spaces $\mathbb{L}_{H,i}^{1,p}$. That is,
$\mathbb{L}_{H,i}^{1,p}$ is the set of processes $u\in\mathbb{D}_{i}%
^{1,p}(\left|  \mathcal{H}\right|  \mathcal{)}$ such that
\[
\left\|  u\right\|  _{\mathbb{L}_{H,i}^{1,p}}^{p}:=\mathbb{E}\left[  \left\|
u\right\|  _{L^{1/H}\left(  \left[  0,T\right]  \right)  }^{p}\right]
+\mathbb{E}\left[  \left\|  D^{(i)}u\right\|  _{L^{1/H}\left(  \left[
0,T\right]  ^{2}\right)  }^{p}\right]  <\infty\text{.}%
\]
The processes $u\in\mathbb{L}_{H,i}^{1,p}$ are, in general, functionals of all
the components of the process $B$.

If a process $u=\{ u_{t}, t\in[0,T]\}$ belongs to the domain of $\delta^{(i)}%
$, we call $\delta^{(i)} (u)$ the Skorohod integral of $u$ with respect to the
fractional Brownian motion $B^{(i)}$ and we will make use of the notation
\[
\delta^{(i)} (u)= \int_{0}^{T} u_{s}\delta B^{(i)}_{s}.
\]

\section{Stochastic calculus for the fBm}

There exist a huge literature on the stochastic calculus for the fBm. There
are essentially two types of stochastic integrals: path-wise integrals defined
using the Young's integral in the case $H>1/2$, and the Skorohod integral
which is the adjoint of the divergence operator introduced in the previous
section. We refer the reader to the monographs by Biagini, Hu, \O ksendal and
Zhang \cite{BiHuOkZh08} and Mishura \cite{Mi08} and the references therein.

In order to describe the evolution of the eigenvalues of a matrix fractional
Brownian motion we need a multidimensional version of the It\^o formula for
the Skorohod integral, in the case of functions that are smooth only on a
dense open subset of the Euclidean space and satisfy some growth requirements.
This type of formula is not covered by the existing literature on the subject
and we provide below a proof based on a duality argument, following the
approach developed in the paper \cite{CheNu05}.

\begin{theorem}
\label{Th41}Suppose that $B$ is an $N$-dimensional fractional Brownian motion
with Hurst parameter $H>1/2$. Consider a function $F:\mathbb{R}^{N}%
\rightarrow\mathbb{R}$ such that:

\begin{enumerate}
\item There exists an open set $G\subset\mathbb{R}^{N}$ such that $G^{c}$ has
zero Lebesgue measure and $F$ is twice continuously differentiable in $G$.

\item $|F(x)| + \left|  \frac{\partial F}{\partial x_{i}}(x) \right|  \le
C(1+|x|^{M})$, for some constants $C>0$ and $M>0$ and for all $x\in G$ and
$i=1,\dots, N$.

\item For each $i=1,\dots, N$ and for each $s>0$ and $p\ge1$,
\[
\mathbb{E} \left[  \left|  \frac{\partial^{2} F}{\partial x_{i}^{2}} (B_{s})
\right|  ^{p} \right]  \le Cs^{-pH},
\]
for some constant $C>0$.
\end{enumerate}

Then, for each $i=1,\dots,N$ and $t\in[0,T]$, the process $\{ \frac{\partial
F}{\partial x_{i}}(B_{s})\mathbf{1}_{[0,t]}(s), s\in[0,T] \}$ belongs to the
space $\mathbb{L}_{H,i}^{1,1/H}$ and
\begin{equation}
\label{eq1}F(B_{t})= F(0)+ \sum_{i=1}^{N} \int_{0}^{t} \frac{\partial
F}{\partial x_{i}} (B_{s}) \delta B^{i}_{s}+ H \sum_{i=1}^{N} \int_{0}^{t}
\ \frac{\partial^{2} F}{\partial x_{i}^{2}} (B_{s}) s^{2H-1} ds.
\end{equation}

\end{theorem}

\begin{proof}
Notice first that the processes $F(B_{t})$, $\frac{\partial F}{\partial x_{i}%
}(B_{t})$ and $\frac{\partial^{2}F}{\partial x_{i}^{2}}(B_{t})$ are well
defined because the probability that $B_{t}$ belongs to $G^{c}$ is zero. On
the other hand, $F(0)$ is also well defined as the limit in $L^{1}(\Omega)$ of
$F(B_{t})$ as $t$ tends to zero. This limit exists because for any $s<t$ we
can write
\begin{align*}
|F(B_{t})-F(B_{s})| &  =\left\vert \int_{0}^{1}\sum_{i=1}^{N}\frac{\partial
F}{\partial x_{i}}(B_{s}+\sigma(B_{t}-B_{s}))(B_{t}^{i}-B_{s}^{i}%
)d\sigma\right\vert \\
&  \leq C\left(  1+\frac{1}{M+1}|B_{t}-B_{s}|^{M}\right)  .
\end{align*}
Conditions (2) and (3) imply that for each $i=1,\dots,N$, the process
$u_{i}(s)=\frac{\partial F}{\partial x_{i}}(B_{s})\mathbf{1}_{[0,t]}(s)$
belongs to the space $\mathbb{L}_{H,i}^{1,1/H}$. In fact,
\[
\mathbb{E}\left[  \int_{0}^{T}\left\vert u_{i}(s)\right\vert ^{1/H}ds\right]
\leq C^{1/H}\mathbb{E}\left[  \int_{0}^{T}(1+|B_{s}|^{M})^{1/H}ds\right]
<\infty,
\]
and
\[
\mathbb{E}\left[  \int_{0}^{T}\int_{0}^{T}|D_{r}^{(i)}u_{i}(s)|^{\frac{1}{H}%
}drds\right]  =\mathbb{E}\left[  \int_{0}^{T}s\left\vert \frac{\partial^{2}%
F}{\partial x_{i}^{2}}(B_{s})\right\vert ^{\frac{1}{H}}ds\right]  \leq CT.
\]
On the other hand, taking $p=1$ in condition (3), we also have for each
$i=1,\dots,N$,
\[
\mathbb{E}\left[  \int_{0}^{t}\left\vert \frac{\partial^{2}F}{\partial
x_{i}^{2}}(B_{s})\right\vert s^{2H-1}ds\right]  <\infty.
\]
As a consequence, all terms in equation (\ref{eq1}) are well defined. Then, to
prove the equality it suffices to show that for any random variable of the
form
\[
G=g(B_{t_{1}},\dots,B_{t_{n}}),
\]
where $0<t_{1}<\cdots<t_{i}<t<t_{i+1}<\cdots t_{n}$, for some $i=0,1,\dots,n$,
where $g$ is $C^{\infty}$ with compact support, we have
\begin{equation}
\mathbb{E}\left[  GF(B_{t})-GF(0)\right]  =\sum_{i=1}^{N}\mathbb{E}\left[
G\int_{0}^{t}\frac{\partial F}{\partial x_{i}}(B_{s})\delta B_{s}^{i}\right]
+H\sum_{i=1}^{N}\mathbb{E}\left[  G\int_{0}^{t}\ \frac{\partial^{2}F}{\partial
x_{i}^{2}}(B_{s})s^{2H-1}ds\right]  .\label{pon}%
\end{equation}
Denote by $p_{t,\tau}(y,x)$ the joint density of the vector $(B_{t_{1}}%
,\dots,B_{t_{n}},B_{t})$, where $\tau=(t_{1},\dots,t_{n})$. By the duality
between the Skorohod integral and the derivative operator (see (\ref{dua})),
we can write
\begin{align*}
\mathbb{E}\left[  G\int_{0}^{t}\frac{\partial F}{\partial x_{i}}(B_{s})\delta
B_{s}^{i}\right]   &  =\mathbb{E}\left[  \left\langle D^{i}G,\frac{\partial
F}{\partial x_{i}}(B_{\cdot})\mathbf{1}_{[0,t]}\right\rangle _{\mathcal{H}%
}\right]  \\
&  =\sum_{j=1}^{n}\mathbb{E}\left[  \frac{\partial g}{\partial y_{ij}%
}(B_{t_{1}},\dots,B_{t_{n}})\left\langle \mathbf{1}_{[0,t_{j}]},\frac{\partial
F}{\partial x_{i}}(B_{\cdot})\mathbf{1}_{[0,t]}\right\rangle _{\mathcal{H}%
}\right]  \\
&  =\int_{0}^{t}\sum_{j=1}^{n}\mathbb{E}\left[  \frac{\partial g}{\partial
y_{ij}}(B_{t_{1}},\dots,B_{t_{n}})\frac{\partial F}{\partial x_{i}}%
(B_{s})\right]  \frac{\partial R}{\partial s}(s,t_{j})ds\\
&  =\int_{0}^{t}\sum_{j=1}^{n}\int_{\mathbb{R}^{N(n+1)}}\frac{\partial
g}{\partial y_{ij}}(y)\frac{\partial F}{\partial x_{i}}(x)\frac{\partial
R}{\partial s}(s,t_{j})p_{s,\tau}(x,y)dxdyds.
\end{align*}
We can integrate by parts in the above expression, for each fixed
$s\not \in \{0,t_{1},\dots,t_{n}\}$. We know that $\frac{\partial F}{\partial
x_{i}}$ is only differentiable in $G$, but using condition (2), and a
regularization procedure, we can proof rigorously this integration by parts
argument. In that way we obtain
\[
\mathbb{E}\left[  G\int_{0}^{t}\frac{\partial F}{\partial x_{i}}(B_{s})\delta
B_{s}^{i}\right]  =-\int_{0}^{t}\sum_{j=1}^{n}\int_{\mathbb{R}^{N(n+1)}}%
\frac{\partial g}{\partial y_{ij}}(y)F(x)\frac{\partial R}{\partial s}%
(s,t_{j})\frac{\partial p_{s,\tau}}{\partial x_{i}}(x,y)dxdyds.
\]
On the other hand,
\begin{align*}
\mathbb{E}\left[  GF(B_{t})-GF(0)\right]   &  =\int_{\mathbb{R}^{N(n+1)}%
}F(x)g(y)p_{t,\tau}(x,y)dxdy-\mathbb{E}(G)F(0)\\
&  =\int_{0}^{t}\int_{\mathbb{R}^{N(n+1)}}F(x)g(y)\frac{\partial p_{s,\tau}%
}{\partial s}(x,y)dxdyds.
\end{align*}
The second term in (\ref{pon}) follows by regularizing $F$ with an
approximation of the identity. Also, integrating by parts,
\begin{align*}
\mathbb{E}\left[  G\int_{0}^{t}\frac{\partial^{2}F}{\partial x_{i}^{2}}%
(B_{s})s^{2H-1}ds\right]   &  =\int_{0}^{t}g(y)\frac{\partial^{2}F}{\partial
x_{i}^{2}}(x)p_{s,\tau}(x,y)s^{2H-1}dxdyds\\
&  =\int_{0}^{t}g(y)F(x)\frac{\partial^{2}p_{s,\tau}}{\partial x_{i}^{2}%
}(x,y)s^{2H-1}dxdyds.
\end{align*}
Finally, the result follows from
\[
\frac{\partial p_{s,\tau}}{\partial s}=\sum_{i=1}^{N}\sum_{j=1}^{n}%
\frac{\partial R}{\partial s}(s,t_{j})\frac{\partial^{2}p_{s,\tau}}{\partial
y_{ij}\partial x_{i}}+\sum_{i=1}^{N}Hs^{2H-1}\frac{\partial^{2}p_{s,\tau}%
}{\partial x_{i}^{2}}.
\]

\end{proof}

In the rest of this section we will recall a property on the $1/H$-variation
of the divergence integral established in \cite{GN05}. Fix $T>0$ and set
$t_{i}^{n}:=\frac{iT}{n}$, where $n$ is a positive integer and $i=0,1,\ldots
,n$. Given a stochastic process $X=\{ X_{t}, t\in[ 0,T]\} $ we define for each
$p\ge1$,
\[
V_{n}^{p}(X):=\sum_{i=0}^{n-1}\left|  X_{t_{i+1}^{n}}-X_{t_{i}^{n}}\right|
^{p}.
\]

\begin{definition}
\label{def1} The $p$-variation, $p\ge1$, of a stochastic process $X$ is
defined as the limit, in $L^{1}\left(  \Omega\right)  $, if it exists, of
$V_{n}^{p}(X)$ as $n\rightarrow\infty$.
\end{definition}

Then, following result (see Theorem 4.8 in \cite{GN05}), allows one to compute
the $1/H$-variation of a Skorohod integral with respect to a multidimensional
fractional Brownian motion.

\begin{theorem}
\label{Theorem multidimensional} Let $\frac{1}{2}<H<1$ and $u^{i}\in
\mathbb{L}_{H,i}^{1,1/H}$ for each $i=1,\ldots N$. Set $X_{t}:=$ $\sum
_{i=1}^{N}\int_{0}^{t}u_{s}^{\left(  i\right)  }\delta B_{s}^{\left(
i\right)  }$, for each $t\in\left[  0,T\right]  $. Then
\[
V_{n}^{1/H}(X)\underset{n\rightarrow\infty}{\overset{L^{1}(\Omega
)}{\longrightarrow}}\int_{\mathbb{R}^{N}}\left[  \int_{0}^{T}\left\vert
\left\langle u_{s},\xi\right\rangle \right\vert ^{1/H}\mathrm{d}s\right]
\nu\left(  d\xi\right)  ,
\]
where $\nu$ is the standard normal distribution on $\mathbb{R}^{N}$.
\end{theorem}

\section{No colliding of eigenvalues of a matrix Gaussian H\"{o}lder
continuous}

In this section we will show that for the random symmetric matrix
corresponding to a general Gaussian process with H\"older continuous
trajectories of order larger than $1/2$, the eigenvalues do not collide almost surely.

Suppose that $x=\{x(t),t\geq0\}$ is a zero mean Gaussian process satisfying
\begin{equation}
\mathbb{E}(|x(t)-x(s)|^{2})\leq C_{T}|t-s|^{2\gamma}, \label{CovIneq}%
\end{equation}
for any $s,t\in\lbrack0,T]$, where $\gamma\in(1/2,1)$. Suppose also that
$x(0)=0$ and the variance of $x(t)$ is positive for any $t>0$. We know that,
by Kolmogorov continuity theorem, the trajectories of $x$ are H\"{o}lder
continuous of order $\beta$ for any $\beta<\gamma$.

Consider a symmetric random matrix of the form
\[
X(t)=X(0)+\widehat{X}(t),
\]
where $X(0)$ is a fixed deterministic symmetric matrix and $\widehat{X}%
_{ij}(t)=x_{ij}(t)$ if $i<j$ and $\widehat{X}_{ii}(t)=\sqrt{2}x_{i,i}(t)$,
where $\{x_{ij},i\leq j \}$ are independent copies of the process $x$. For any
$t>0$ the matrix $X(t)$ has full rank a.s. The following is the main result of
this section.

\begin{theorem}
\label{ThNC} Denote by $\lambda_{i}(t)$ the eigenvalues of the random matrix
$X(t)$, $i=1,\dots,n$. We can assume that $\lambda_{1}(t)\geq\cdots\geq
\lambda_{d}(t)$. Then,
\begin{equation}
P(\lambda_{1}(t)>\cdots>\lambda_{d}(t),\forall t>0)=1. \label{NCExp}%
\end{equation}

\end{theorem}

\begin{proof}
The proof will be done in several steps. Fix $t_{0}>0$. From the well-known
results for the Gaussian Orthogonal random matrix we know that $\lambda
_{1}(t_{0})>\cdots>\lambda_{d}(t_{0})$ almost surely.

Applying the Hoffman-Weilandt inequality (see \cite{HoW53}), we deduce
\[
\sum_{i=1}^{d}(\lambda_{i}(t)-\lambda_{i}(s))^{2}\leq\frac{1}{d}\sum
_{i,j=1}^{d}(X_{ij}(t)-X_{ij}(s))^{2}=\frac{1}{d}\sum_{i,j=1}^{d}%
(\widehat{X}_{ij}(t)-\widehat{X}_{ij}(s))^{2}%
\]
for any $s,t \ge0$. This implies that for each $i$ and each real $p\ge1$,
\begin{equation}
\label{Hol}\mathbb{E}(|\lambda_{i}(t)-\lambda_{i}(s)|^{p})\leq C|t-s|^{p\gamma
},
\end{equation}
and choosing $p$ such that $p\gamma>1$ we deduce that the trajectories of
$\lambda_{i}(t)$ are H\"{o}lder continuous of order $\beta$ for any
$\beta<\gamma$.

Consider the stopping time
\[
\tau=\inf\{t\geq t_{0}:\lambda_{i}(t)=\lambda_{j}(t) \,\,\,
\hbox{for some}\,\,\, i\not =j\}.
\]
Notice that $\tau>t_{0}$ almost surely. On the random interval $[t_{0},\tau)$
the function $\log(\lambda_{i}(t)-\lambda_{j}(t)$, where $i\not =j$, is well
defined and we can use the stochastic calculus with respect to the Young's
integral to write for any $t_{0}\le t<\tau$,
\begin{equation}
\log(\lambda_{i}(t)-\lambda_{j}(t))=\log(\lambda_{i}(t_{0})-\lambda_{j}%
(t_{0}))+\int_{t_{0}}^{t}\frac{1}{\lambda_{i}(s)-\lambda_{j}(s)}\left(
d\lambda_{i}(s)-d\lambda_{j}(s)\right)  . \label{41}%
\end{equation}
The Riemann-Stieltjes integral
\[
I^{i}_{i,j}(t):=\int_{t_{0}}^{t}\frac{1}{\lambda_{i}(s)-\lambda_{j}%
(s)}d\lambda_{i}(s),
\]
can be expressed in terms of fractional derivative operators, following the
approach by Z\"able \cite{Za08}. Choosing $\alpha$ such that $1-\gamma
<\alpha<\frac{1}{2}$, we obtain
\begin{align*}
I^{i}_{i,j}(t) =\int_{t_{0}}^{t}D_{t_{0}+}^{\alpha}(\lambda_{i}-\lambda
_{j})_{t_{0}}^{-1}(s)D_{t-}^{1-\alpha}\lambda_{i,t-}(s)ds +(\lambda_{i}%
(t_{0})-\lambda_{j}(t_{0}))^{-1}\left(  \lambda_{i}(t)-\lambda_{i}%
(t_{0}))\right)  ,
\end{align*}
where
\begin{align*}
I_{i,j}(s):=D_{t_{0}+}^{\alpha}(\lambda_{i}-\lambda_{j})_{0}^{-1}(s)  &
=\frac{1}{\Gamma\left(  1-\alpha\right)  }\Bigg (s^{-\alpha}\left(  \frac
{1}{\lambda_{i}(s)-\lambda_{j}(s)}-\frac{1}{\lambda_{i}(t_{0})-\lambda
_{j}(t_{0})}\right) \\
&  +\alpha\int_{t_{0}}^{s}\frac{(\lambda_{i}(s)-\lambda_{j}(s))^{-1}%
-(\lambda_{i}(y)-\lambda_{j}(y))^{-1}}{\left(  s-y\right)  ^{\alpha+1}%
}dy\Bigg),
\end{align*}
and
\[
J_{j}(s):= D_{t-}^{1-\alpha}\lambda_{i,t-}(s)=\frac{1}{\Gamma\left(
\alpha\right)  }\left(  \frac{\lambda_{i}\left(  s\right)  -\lambda_{i}%
(t)}{\left(  t-s\right)  ^{1-\alpha}}+(1-\alpha)\int_{s}^{t}\frac{\lambda
_{i}\left(  s\right)  -\lambda_{i}\left(  y\right)  }{\left(  y-s\right)
^{2-\alpha}}dy\right)  .
\]
We claim that
\begin{equation}
\label{claim}P\left(  \int_{t_{0}} ^{t} |I_{i,j}(s)||J_{j}(s)| ds <\infty,
\,\, \hbox{for all} \,\,\, t\ge t_{0}, \,\, i\not =j \right)  =1.
\end{equation}
This claim implies that for all $i\not =j$, $\int_{t_{0}} ^{\tau}%
|I_{i,j}(s)||J_{j}(s)| ds <\infty$ almost surely on the set $\{\tau<\infty\}$.
Therefore $P(\tau=\infty)=1$, otherwise we would get a contradiction with
$\log(\lambda_{i}(\tau) -\lambda_{j}(\tau)) =-\infty$.

In order to prove the claim (\ref{claim}) we are going to show that for all
$t\geq t_{0}$, and for all $i\not =j$,
\begin{equation}
\mathbb{E}\left(  \int_{t_{0}}^{t}|I_{i,j}(s)||J_{j}(s)|\right)  ds<\infty.
\label{claim2}%
\end{equation}
In order to show (\ref{claim2}), we first apply H\"{o}lder's inequality with
exponents $p,q>1$ such that $\frac{1}{p}+\frac{1}{q}=1$ and we get
\[
\mathbb{E}(|I_{i,j}(s)||J_{j}(s)|)\leq\lbrack E(|I_{i,j}(s)|^{p})]^{\frac
{1}{p}}[E(|J_{j}(s)|^{q})]^{\frac{1}{q}}.
\]
By the estimate (\ref{Hol}) for any fixed $\beta$ such that $1-\alpha
<\beta<\gamma$ there exists a random variable $G$ with moments of all orders
such that for all $i$
\begin{equation}
|\lambda_{i}(s)-\lambda_{i}(y)|\leq CG|s-y|^{\beta}, \label{Hol2}%
\end{equation}
for all $s,y\in\lbrack t_{0},t]$, which leads to the estimate
\begin{equation}
E(|J_{j}(s)|^{q})\leq C_{q}E(G^{q}), \label{J}%
\end{equation}
for all $q>1$ and for some constant $C>0$. In order to estimate $E(|I_{i,j}%
(s)|^{p})$ we consider first the second summand in the definition of
$I_{i,j}(s)$ denoted by
\[
K_{i,j}(s):=\int_{t_{0}}^{s}\frac{(\lambda_{i}(s)-\lambda_{j}(s))^{-1}%
-(\lambda_{i}(y)-\lambda_{j}(y))^{-1}}{\left(  s-y\right)  ^{\alpha+1}}dy.
\]
This term can be expressed as
\[
K_{i,j}(s)=\int_{t_{0}}^{s}\frac{(\lambda_{i}(y)-\lambda_{j}(y)-\lambda
_{i}(s)+\lambda_{j}(s))}{\left(  s-y\right)  ^{\alpha+1}[(\lambda
_{i}(s)-\lambda_{j}(s)][(\lambda_{i}(y)-\lambda_{j}(y))]}dy.
\]
Then, if $a+b=1$, using the estimate (\ref{Hol2}), we obtain
\begin{align*}
|K_{i,j}(s)|  &  \leq\int_{t_{0}}^{s}\frac{|\lambda_{i}(y)-\lambda
_{j}(y)-\lambda_{i}(s)+\lambda_{j}(s)|^{a}|\lambda_{i}(y)-\lambda
_{j}(y)-\lambda_{i}(s)+\lambda_{j}(s)|^{b}}{\left(  s-y\right)  ^{\alpha
+1}|(\lambda_{i}(s)-\lambda_{j}(s)||(\lambda_{i}(y)-\lambda_{j}(y))|}dy\\
&  \leq(2G)^{a}\int_{t_{0}}^{s}\frac{|\lambda_{i}(y)-\lambda_{j}%
(y)-\lambda_{i}(s)+\lambda_{j}(s)|^{b}}{|(\lambda_{i}(s)-\lambda
_{j}(s)||(\lambda_{i}(y)-\lambda_{j}(y))|}(s-y)^{a\beta-\alpha-1}dy\\
&  \leq(2G)^{a}\int_{t_{0}}^{s}\Big(|\lambda_{i}(y)-\lambda_{j}(y)|^{b-1}%
|(\lambda_{i}(s)-\lambda_{j}(s)|^{-1}\\
&  +|\lambda_{i}(y)-\lambda_{j}(y)|^{-1}|(\lambda_{i}(s)-\lambda_{j}%
(s)|^{b-1}\Big)(s-y)^{a\beta-\alpha-1}dy.
\end{align*}
Therefore
\begin{align*}
\Vert K_{i,j}(s)\Vert_{p}  &  \leq2^{a}\Vert G^{a}\Vert_{p_{1}}\int_{t_{0}%
}^{s}\Big(\Vert|\lambda_{i}(y)-\lambda_{j}(y)|^{b-1}\Vert_{p_{2}}\Vert
|\lambda_{i}(s)-\lambda_{j}(s)|^{-1}\Vert_{p_{3}}\\
&  +\Vert|\lambda_{i}(y)-\lambda_{j}(y)|^{-1}\Vert_{p_{3}}\Vert|\lambda
_{i}(s)-\lambda_{j}(s)|^{b-1}\Vert_{p_{2}}\Big)(s-y)^{a\beta-\alpha-1}dy,
\end{align*}
where $\frac{1}{p}=\frac{1}{p_{1}}+\frac{1}{p_{2}}+\frac{1}{p_{3}}$, with
$p_{i}>1$ for $i=1,2,3$. We choose $a,p_{1},p_{2}$ and $p_{3}$ such that
\[
a>\frac{\alpha}{\beta},\quad p_{3}<2,\quad p_{2}<\frac{2\beta}{\alpha},
\]
which is possible by taking $p$ and $p_{1}$ close to $1$ and using the
inequality $\alpha<\frac{1}{2}<\beta$. Finally, to complete the proof we need
to estimate the expectation
\begin{equation}
\mathbb{E}(|(\lambda_{i}(s)-\lambda_{j}(s)|^{-q}) \label{4.1}%
\end{equation}
when $q<2$. The joint density of the eigenvalues $\lambda_{1}(s)>\cdots
>\lambda_{d}(s)$ is given by%

\begin{equation}
c_{d}\prod_{k<h}\left\vert \lambda_{k}-\lambda_{h})\right\vert \sigma
(s)^{-d(d+1)/2}\exp\left(  -\sum_{i=1}^{d}\frac{\lambda_{i}^{2}}{4\sigma
^{2}(s)}\right)  , \label{eEOrht}%
\end{equation}
where $c_{d}$ is a constant depending only on $d$ and $\sigma^{2}(s)$ is the
variance of $x(s);$ see \cite[Th. 2.5.2]{AGZ09}. Then, the expectation
$\mathbb{E}(|(\lambda_{i}(s)-\lambda_{j}(s)|^{-q})$ can be estimated up to a
constant by
\[
\int_{\mathbb{R}^{d}}\prod_{k<h}(\lambda_{k}-\lambda_{h})|\lambda_{i}%
-\lambda_{j}|^{-q}\sigma(s)^{-d(d+1)/2}\exp\left(  -\sum_{i=1}^{d}%
\frac{\lambda_{i}^{2}}{4\sigma^{2}(s)}\right)  d\lambda.
\]
Making the change of variable $\lambda_{i}=\sigma(s)\mu_{i}$ we get
\[
\mathbb{E}(|(\lambda_{i}(s)-\lambda_{j}(s)|^{-q})\leq C\sigma(s)^{-q}.
\]
Our hypothesis of positivity of the variance, together with its continuity,
imply that $\sigma(s)$ is bounded away from zero in the interval $[t_{0},t]$.
Therefore, $\mathbb{E}(|(\lambda_{i}(s)-\lambda_{j}(s)|^{-q})$ is uniformly
bounded on $[t_{0},t]$. This allows us to complete the proof of (\ref{claim2}%
). So, the claim (\ref{claim}) holds, and this implies that $P(\tau=\infty
)=1$. Finally, letting $t_{0}$ tend to zero we obtain the desired result.
\end{proof}

\begin{remark}
\label{GU} Theorem \ref{ThNC} also holds in the case of a random Hermitian
matrix corresponding to a general Gaussian process with H\"{o}lder continuous
trajectories of order $\gamma.$ Namely, consider an Hermitian random matrix
\[
X(t)=X(0)+\widehat{X}(t),
\]
where $X(0)$ is a fixed deterministic Hermitian matrix and
\[
\widehat{X}_{ij}(t)=%
\begin{cases}
\frac{1}{\sqrt{2}}\left\{  \mathrm{Re}(x_{ij}(t))+\sqrt{-1}\mathrm{Im}%
(x_{ij}(t))\right\}  & \mbox{if}\,\,i\neq j,\\
x_{ii}(t) & \mbox{if}\,\,i=j,
\end{cases}
\]
where $\{\mathrm{Re}(x_{ij}(t)),\mathrm{Im}(x_{ij}(t)),i<j,x_{ii}(t)\}$ are
independent copies of a zero mean Gaussian process $x=\{x(t),t\geq0\}$
satisfying (\ref{CovIneq}) for any $s,t\in\lbrack0,T]$, where $\gamma
\in(1/2,1)$. Since for each $t>0$ $X(t)$ is a Gaussian unitary ensemble, then
the matrix $X(t)$ has full rank and distinct eigenvalues a.s. In this case the
joint density of the eigenvalues $\lambda_{1}(s)>\cdots>\lambda_{d}(s)$ is
given by (see \cite[Th. 2.5.2]{AGZ09})%
\begin{equation}
\widetilde{c}_{d}\prod_{k<h}\left\vert \lambda_{k}-\lambda_{h}\right\vert
^{2}\sigma(s)^{-d^{2}}\exp\left(  -\sum_{i=1}^{d}\frac{\lambda_{i}^{2}%
}{2\sigma^{2}(s)}\right)  , \label{EinHer}%
\end{equation}
where $\widetilde{c}_{d}$ is a constant depending only on $d$ and $\sigma
^{2}(s)$ is the variance of $x(s),$ which is assumed to be positive. Then,
proceeding as in the proof of Theorem \ref{ThNC} one can prove that if
$\lambda_{i}(t)$ denote the eigenvalues of the random matrix $X(t)$,
$i=1,\dots,n$ and $\lambda_{1}(t)\geq\cdots\geq\lambda_{d}(t)$, then
(\ref{NCExp}) also holds.
\end{remark}

\section{Stochastic differential equation for the eigenvalues of a matrix fbm}

We first consider several needed facts about eigenvalues as functions of
entries of a symmetric matrix. We denote by $\mathcal{H}_{d}$ the collection
of symmetric $d$-dimensional matrices. For a matrix $X=(x_{ij})\in
\mathcal{H}_{d}$ we use the coordinates $x_{ij},i\leq j$, and in this way we
identify $\mathcal{H}_{d}$ with $\mathbb{R}^{d(d+1)/2}$. We denote by
$\mathcal{H}_{d}^{vg}$ the set of matrices $X\in\mathcal{H}_{d}$ such that
there is a factorization
\[
X=UDU^{\ast},
\]
where $D$ is a diagonal matrix with entries $\lambda_{i}=D_{ii}$ such that
$\lambda_{1}>\lambda_{2}>\cdots>\lambda_{d}$, $U$ is an orthogonal matrix,
with $U_{ii}>0$ for all $i$, $U_{ij}\not =0$ for all $i,j$ and all minors of
$U$ have non zero determinants. The matrices in the set $\mathcal{H}_{d}^{vg}$
are called \textit{very good} matrices, and we can identify $\mathcal{H}%
_{d}^{vg}$ as an open subset of $\mathbb{R}^{d(d+1)/2}$. It is known that the
complement of $\mathcal{H}_{d}^{vg}$ has zero Lebesgue measure.

Denote by $\mathcal{U}_{d}^{vg}$ the set of all orthogonal matrices $U$, with
$U_{ii}>0$ for all $i$, $U_{ij}\not =0$ for all $i,j$ and all minors of $U$
have non zero determinants. Let $\mathcal{S}_{d}$ be the open simplex
\begin{equation}
\mathcal{S}_{d}=\{(\lambda_{1},\dots,\lambda_{d})\in\mathbb{R}^{d}:\lambda
_{1}>\lambda_{2}>\cdots>\lambda_{d}\}.\label{OpSimpx}%
\end{equation}
For any $\lambda=(\lambda_{1},\dots,\lambda_{d})\in\mathcal{S}_{d}$, let
$D_{\lambda}$ the diagonal matrix such that $D_{ii}=\lambda_{i}$. On the other
hand, consider the mapping $T:\mathcal{U}_{d}^{vg}\rightarrow\mathbb{R}%
^{d(d-1)/2}$ defined by
\[
T(U)=\left(  \frac{U_{12}}{U_{11}},\dots,\frac{U_{1d}}{U_{11}},\dots
,\frac{U_{d-1,d}}{U_{d-1,d-1}}\right)  .
\]
It is known that $T$ is bijective and smooth. Then, the mapping $\hat
{T}:\mathcal{S}_{d}\times T(\mathcal{U}_{d}^{vg})\rightarrow\mathcal{H}%
_{d}^{vg}$ given by $\hat{T}(\lambda,z)=T^{-1}(z)D_{\lambda}T^{-1}(z)^{\ast}$
is a smooth bijection. Denote by $\Psi$ the inverse of $\hat{T}$. Then,
\[
\Psi(X)=(\Phi(X),T(U)).
\]
As a consequence of these results, $\lambda(X)=\Phi(X)$ is a smooth function
of $X\in\mathcal{H}_{d}^{vg}$.

Suppose that $X$ is a smooth function of a parameter $\theta\in\mathbb{R}$.
Then, we know that
\[
\partial_{\theta}\lambda_{i}=(U^{\ast}\partial_{\theta}XU)_{ii}%
\]
and
\[
\partial_{\theta}^{2}\lambda_{i}=(U^{\ast}\partial_{\theta}^{2}XU)_{ii}%
+2\sum_{j\not =i}\frac{|(U^{\ast}\partial_{\theta}XU)_{ij}|^{2}}{\lambda
_{i}-\lambda_{j}}.
\]
In particular if $\theta=x_{kh}$ with $k\leq h$, then
\[
\frac{\partial\lambda_{i}}{\partial x_{kh}}=2U_{ik}U_{ih}\mathbf{1}%
_{\{k\not =h\}}+U_{ik}^{2}\mathbf{1}_{\{k=h\}},
\]
and
\[
\frac{\partial^{2}\lambda_{i}}{\partial x_{kh}^{2}}=2\sum_{j\not =i}%
\frac{|U_{ik}U_{jh}+U_{ih}U_{jk}|^{2}}{\lambda_{i}-\lambda_{j}}\mathbf{1}%
_{\{k\not =h\}}+2\sum_{j\not =i}\frac{|U_{ik}U_{jk}|^{2}}{\lambda_{i}%
-\lambda_{j}}\mathbf{1}_{\{k=h\}}.
\]

Consider now a family of independent fractional Brownian motions with Hurst
parameter $H\in(1/2,1)$, $b=\{\{b_{ij}(t),t\geq0$\}$,1\leq i\leq j\leq d\}$.
We define the \textit{symmetric matrix fractional Brownian motion (with
parameter} $H)$ $B(t)$ by $B_{i,j}(t)=b_{ij}(t)$ if $i<j$ and $B_{ii}%
(t)=\sqrt{2}b_{i,i}(t)$. We identify $B(t)$ as an element in $\mathbb{R}%
^{d(d+1)/2}$.

As a consequence of the previous discussion, we have the following result.

\begin{lemma}
\label{LemSmooth} For any $i=1,\dots,d$, there exists a function $\Phi
_{i}:\mathbb{R}^{d(d+1)/2}\rightarrow\mathbb{R}$, which is $C^{\infty}$ in an
open subset $G\subset\mathbb{R}^{d(d+1)/2}$, with $|G^{c}|=0$, such that
$\lambda_{i}(t)=\Phi_{i}(b(t))$. Moreover, for any $k\leq h$
\begin{equation}
\frac{\partial\Phi_{i}}{\partial b_{kh}}=2U_{ik}U_{ih}\mathbf{1}%
_{\{k\not =h\}}+\sqrt{2}U_{ik}^{2}\mathbf{1}_{\{k=h\}}, \label{tt1}%
\end{equation}
and
\[
\frac{\partial^{2}\Phi_{i}}{\partial b_{kh}^{2}}=2\sum_{j\not =i}\frac
{|U_{ik}U_{jh}+U_{ih}U_{jk}|^{2}}{\lambda_{i}-\lambda_{j}}\mathbf{1}%
_{\{k\not =h\}}+4\sum_{j\not =i}\frac{|U_{ik}U_{jk}|^{2}}{\lambda_{i}%
-\lambda_{j}}\mathbf{1}_{\{k=h\}}.
\]

\end{lemma}

On the other hand, the joint density of the eigenvalues $\lambda_{1}%
(t)>\cdots>\lambda_{d}(t)$ of $B(t)$ can be obtained from \cite[Th.
2.5.2]{AGZ09} as follows%

\begin{equation}
c_{d}\prod_{k<h}\left\vert \lambda_{k}-\lambda_{h}\right\vert t^{-Hd(d+1)/2}%
\exp\left(  -\sum_{i=1}^{d}\frac{\lambda_{i}^{2}}{4t^{2H}}\right)  ,
\label{defbm}%
\end{equation}
where $c_{d}$ is a constant depending only on $d.$

Then we can prove the following analogous of (\ref{21}) for the evolving of
the eigenvalues processes of a matrix fractional Brownian motion. It is given
in terms of the Skorohod integral of the functions $\frac{\partial\Phi_{i}%
}{\partial b_{kh}},i=1,...,d.$

\begin{theorem}
\label{thm4.1} Let $H\in(1/2,1)$ and $\left\{  B(t),t\geq0\right\}  $ be a
matrix fractional Brownian motion of parameter $H$ as above. Let $X(0)$ be an
arbitrary deterministic symmetric matrix and $B(0)=X(0).$ For each $t\geq0,$
let $\lambda_{1}(t),...,\lambda_{d}(t)$ be the eigenvalues of $B(t).$ Then,
for any $t>0$ and $i=1,\dots,d$,
\begin{equation}
\lambda_{i}(t)=\lambda_{i}(0)+Y_{t}^{i}+2H\sum_{j\not =i}\int_{0}^{t}%
\frac{s^{2H-1}}{\lambda_{i}(s)-\lambda_{j}(s)}ds, \label{31}%
\end{equation}
where
\begin{equation}
Y_{t}^{i}=\sum_{k\leq h}\int_{0}^{t}\frac{\partial\Phi_{i}}{\partial b_{kh}%
}(b(s))\delta b_{kh}(s). \label{32}%
\end{equation}

\end{theorem}

\begin{proof}
Without loss of generality we can consider $X(0)=0$. Consider the function
$\Phi:\mathbb{R}^{d(d+1)/2} \rightarrow\mathbb{R}^{d}$ introduced in Lemma
\ref{LemSmooth}, which is $C^{\infty}$ in an open subset $G\subset
\mathbb{R}^{d(d+1)/2}$ whose complement has zero Lebesgue measure. We claim
that this function satisfies the assumptions of Theorem \ref{Th41}. First
notice that
\[
\sum_{i=1}^{d}\Phi_{i}^{2}\leq\frac{1}{d}\sum_{i=1}^{d}b_{ij}^{2}+\frac{1}%
{2}d\sum_{i<j}^{d}b_{ij}^{2},
\]
and
\[
\left\vert \frac{\partial\Phi_{i}}{\partial b_{kh}}\right\vert \leq2+\sqrt
{2}.
\]
On the other hand, taking into account that $\lambda_{i}(t)=\Phi_{i}(b(t))$
and using the density of the eigenvalues (\ref{defbm}), we can write
\begin{align}
\mathbb{E}\left[  \left\vert \frac{\partial^{2}\Phi_{i}}{\partial b_{kh}^{2}%
}(b_{s})\right\vert ^{p}\right]   &  \leq C_{p}\sum_{j\not =i}\mathbb{E}%
[|\lambda_{i}(s)-\lambda_{j}(s)|^{-p}]\nonumber\\
&  =C_{p}\sum_{j\not =i}\int_{\mathcal{S}_{d}}\prod_{k<h}(\lambda_{k}%
-\lambda_{h})|\lambda_{i}-\lambda_{j}|^{-p}s^{-\frac{Hd(d+1)}{2}}\exp\left(
-\sum_{i=1}^{d}\frac{\lambda_{i}^{2}}{4s^{2H}}\right)  d\lambda\nonumber\\
&  \leq C_{p}s^{-pH}, \label{Est}%
\end{align}
where we have made the change of variable $\lambda_{i}=s^{H}\mu_{i}$.

Therefore, conditions (1), (2) and (3) if Theorem \ref{Th41} hold for $\Phi$.
Notice also that
\[
\sum_{k\leq h}\frac{\partial^{2}\Phi_{i}}{\partial b_{kh}^{2}}=2\sum
_{j\not =i}\frac{1}{\lambda_{i}-\lambda_{j}}.
\]

Then, Theorem \ref{Th41} yields
\[
\lambda_{i}(t)=Y_{t}^{i}+2H\sum_{j\not =i}\int_{0}^{t}\frac{s^{2H-1}}%
{\lambda_{i}(s)-\lambda_{j}(s)}ds,
\]
where $Y_{t}^{i}$ is given by (\ref{32}). We remark that by Theorem \ref{ThNC}
$(\lambda_{1}(t), \dots,\lambda_{d}(t))$ belongs to the open simplex
$\mathcal{S}_{d}$ for each $t>0$ with probability one.
\end{proof}

\begin{remark}
a) Notice that equation (\ref{31}) is similar to the It\^{o} SDE (\ref{21})
satisfied for the eigenvalues of the matrix Brownian motion, in the case
$H=1/2$. A natural question is to ask whether the processes $Y^{i}$ are
one-dimensional independent fractional Brownian motions in the general case
$H>1/2$. In the case $H=1/2$ this is obtained applying L\'{e}vy's
characterization theorem. Unfortunately, although there is a version of this
theorem for the fractional Brownian motion (see \cite{HNS09}), it cannot be
used here due to the lack of martingale properties.

b) We will show below that for each $i=1,\dots,d$, the process $Y^{i}$ has the
same self-similar and variation properties as the fractional Brownian motion.
However, from these properties we cannot conclude that $Y^{i}$ is a fractional
Brownian motion because we do not know if these processes are Gaussian. We
conjecture that these processes are not Gaussian (see \cite{HN05} for a
related problem concerning the fractional Bessel process, where again this
type of non Gaussian process appear).
\end{remark}

\begin{proposition}
Assuming $B(0)=0$, the process $Y=(Y^{1}, \dots, Y^{d})$ is $H$-self-similar.
\end{proposition}

\begin{proof}
Let $a>0$. By the self-similarity of the fractional Brownian motion it follows
that $\{\lambda(at),t\geq0\}$ has the same law as $\{a^{H}\lambda(t),t\geq
0\}$. Then, the result follows from the equation
\[
Y_{t}^{i}=\lambda_{i}(t)-2H\sum_{j\not =i}\int_{0}^{t}\frac{s^{2H-1}}%
{\lambda_{i}(s)-\lambda_{j}(s)}ds,
\]
for $i=1,\dots,d$.
\end{proof}

Finally, as an application of Theorem \ref{Theorem multidimensional}, we show
that for each $i=1,\dots, d$, the process $Y^{i}$ has the same $1/H$ variation
as a fBm with variance $2t^{2H}$.

\begin{proposition}
For each $i=1,\dots, d$, the process $Y^{i}$ has a $1/H$ variation equals to
$\sqrt{2} t \mathbb{E}\left[  |Z| ^{1/H}\right]  $, where $Z$ is $N(0,1)$
random variable.
\end{proposition}

\begin{proof}
By Theorem \ref{thm4.1}, we have that for each $k,h$ the process $\left\{
\frac{\partial\Phi_{i}}{\partial b_{kh}}(b(s)),s\in\lbrack0,T]\right\}  $
belongs to the space $\mathbb{L}_{H,i}^{1,1/H}$ for each $i=1,\ldots
,N=\frac{d(d+1)}{2}$. Therefore, by Theorem \ref{Theorem multidimensional} the
$1/H$-variation of $Y^{i}$ in the time interval $[0,t]$ is given by
\[
\mathbb{E}^{\Theta}\left[  \int_{0}^{t}\left\vert \sum_{k\leq h}\frac
{\partial\Phi_{i}}{\partial b_{kh}}(b(s))\Theta_{kh}\right\vert ^{1/H}%
ds\right]  ,
\]
where $\Theta$ is an $N$-dimensional standard normal random variable. Let us
denote by $S^{N}$ the unitary $N$-dimensional sphere and let $\sigma_{N}$ be
the uniform probability measure defined on $S^{N-1}$. Notice that the vector
$\frac{1}{\sqrt{2}}\left(  \frac{\partial\Phi_{i}}{\partial b_{kh}%
}(b(s))\right)  _{k\leq h}$, denoted by $R_{s}^{i},$ takes values in $S^{N-1}$
because from (\ref{tt1}) we obtain
\[
\sum_{k\leq h}\left\vert \frac{\partial\Phi_{i}}{\partial b_{kh}%
}(b(s))\right\vert ^{2}=2.
\]
Therefore,
\[
2\mathbb{E}^{\Theta}\left[  \int_{0}^{t}\left\vert \langle R_{s}^{i}%
,\Theta\rangle\right\vert ^{1/H}ds\right]  =2\mathbb{E}[|\Theta|]\int%
_{S^{d-1}}\left[  \int_{0}^{T}\left\vert \left\langle R_{s}^{i},\eta
\right\rangle \right\vert ^{1/H}\mathrm{d}s\right]  \sigma_{N}\left(
d\eta\right)  .
\]
Moreover, if $e\in S^{d-1}$, the integral $\int_{S^{N-1}}\left\vert
\left\langle e,\eta\right\rangle \right\vert ^{1/H}\sigma_{N}\left(
d\eta\right)  $ does not depend on the vector $e$. Therefore, choosing
$e=\left(  1,0,\ldots,0\right)  $, yields
\[
2\mathbb{E}^{\Theta}\left[  \int_{0}^{t}\left\vert \langle R_{s}^{i}%
,\Theta\rangle\right\vert ^{1/H}ds\right]  =2t\mathbb{E}[|\Theta
|]\int_{S^{N-1}}\left\vert \eta_{1}\right\vert ^{1/H}\sigma_{N}\left(
d\eta\right)  \newline=2t\mathbb{E}\left[  \left\vert Z\right\vert
^{1/H}\right]  ,
\]
where $Z$ is a one-dimensional $N(0,1)$ random variable. This completes the
proof of the proposition.
\end{proof}

\begin{remark}
Consider the Hermitian matrix fractional Brownian motion $B(t)=$ ($B_{ij}(t))$
where
\[
B_{ij}=%
\begin{cases}
\frac{1}{\sqrt{2}}\mathrm{Re}(b_{ij}(t))+\sqrt{-1}\mathit{\mathrm{Im}}%
(b_{ij}(t)) & \mbox{if}\,\,i\neq j,\\
b_{i,i}(t), & \mbox{if}\,\,i=j,
\end{cases}
\]
where $\left\{  \mathrm{Re}b_{ij}(t)),\mathrm{Im}(b_{ij}(t)),i<j,b_{i,i}%
(t)\right\}  $ is a family of independent fractional Brownian motions with
Hurst parameter $H\in(1/2,1).$ We identify $B(t)$ as an element in
$\mathbb{R}^{d^{2}}$.

a) In this case one can prove that if $X(0)$ is an arbitrary deterministic
Hermitian matrix and $B(0)=X(0)$ and for each $t\geq0,$ $\lambda_{1}%
(t),\dots,\lambda_{d}(t)$ denote the eigenvalues of $B(t)$, then, for any
$t>0$ and $i=1,\dots,d$,
\begin{equation}
\lambda_{i}(t)=\lambda_{i}(0)+Y_{t}^{i}+2H\sum_{j\not =i}\int_{0}^{t}%
\frac{s^{2H-1}}{\lambda_{i}(s)-\lambda_{j}(s)}ds, \label{RU1}%
\end{equation}
where
\begin{equation}
Y_{t}^{i}=\sum_{k\leq h}\int_{0}^{t}\frac{\partial \Upsilon_{i}}{\partial
b_{kh}}(b(s))\delta b_{kh}(s). \label{RU2}%
\end{equation}
The process $Y_{t}^{i}$ \ given by (\ref{RU2}) is such that for each
$i=1,\dots,d$, $\Upsilon_{i}$ is a smooth function of $X\in\mathcal{H}%
_{d}^{(2)vg}$, the set of very good Hermitian matrices consisting of matrices
with a decomposition $X=UDU^{\ast},$ where $D$ is a diagonal matrix with
entries $\lambda_{i}=D_{ii}$ such that $\lambda_{1}>\lambda_{2}>\cdots
>\lambda_{d}$, $U$ is an unitary matrix, with $U_{ii}>0$ for all $i$,
$U_{ij}\not =0$ for all $i,j$ and all minors of $U$ have non zero determinants
($U\in\mathcal{U}_{d}^{(2)vg}$). This follows from Lemma 2.5.6 in \cite{AGZ09}
which gives that the mapping $T^{(2)}:\mathcal{U}_{d}^{(2)vg}\rightarrow
\mathbb{R}^{d(d-1)}$ defined by
\[
T^{(2)}(U)=\left(  \frac{U_{12}}{U_{11}},\dots,\frac{U_{1d}}{U_{11}}%
,\frac{U_{2,3}}{U_{22}},\dots,\frac{U_{2,d}}{U_{2,2}},...\frac{U_{d-1,d}%
}{U_{d-1,d-1}}\right)  .
\]
is bijective and smooth and the complement of $\mathcal{H}_{d}^{(2)vg}$ has
Lebesgue measure zero. Therefore the mapping $\hat{T}^{(2)}:\mathcal{S}%
_{d}\times T^{(2)}(\mathcal{U}_{d}^{(2)vg})\rightarrow\mathcal{H}_{d}^{(2)vg}$
given by $\hat{T}^{(2)}(\lambda,z)=\left(  T^{(2)}\right)  ^{-1}(z)D_{\lambda
}\left(  T^{(2)}\right)  ^{-1}(z)^{\ast}$ is a smooth bijection. As a
consequence, $\lambda(X)=\Upsilon(X)$ is a smooth function of $X\in
\mathcal{H}_{d}^{(2)vg}$.

b) The proof of (\ref{RU1}), analogous to that of Theorem \ref{thm4.1},
requires to consider estimates like (\ref{Est}) but now using the joint
density of the eigenvalues $(\lambda_{1}(t),\dots,\lambda_{d}(t))$ in the
Hermitian case (from (\ref{EinHer})) given by
\[
\widetilde{c_{d}}\prod_{k<h}\left\vert \lambda_{k}-\lambda_{h}\right\vert
^{2}t^{-Hd^{2}}\exp\left(  -\sum_{i=1}^{d}\frac{\lambda_{i}^{2}}{2t^{2H}%
}\right)  ,
\]
where $\widetilde{c_{d}}$ is a constant that depends only on $d.$
\end{remark}

\textbf{Acknowledgment. }The authors would like to thank the referee for
useful comments and suggestions that improved the presentation of the paper.

\end{document}